\documentclass[a4paper]{amsart}
\usepackage{amsmath}
\usepackage{amssymb}
\usepackage{amsfonts}
\usepackage{pxfonts}
\usepackage[OT2,T1]{fontenc}

\usepackage{accents}

\usepackage[
textwidth=14.5cm, 
textheight=21cm,
hmarginratio=1:1,
vmarginratio=1:1]{geometry}

\newtheorem{thm}{Theorem}[section]

\theoremstyle{definition}

\theoremstyle{remark}
\newtheorem{rem}[thm]{Remark}

\numberwithin{equation}{section}
\numberwithin{figure}{section}

\newcommand{\diff}{\mathrm{d}}
\newcommand{\C}{{\mathbb C}}
\newcommand{\R}{{\mathbb R}}

\newcommand{\imag}{\mathrm{i}}
\newcommand{\e}{\mathrm{e}}
\newcommand{\im}{\mathrm{Im}}
\newcommand\vol{\mathrm{vol}}

\newcommand{\Bop}{{\mathbf B}}
\newcommand{\Mop}{{\mathbf M}}

\newcommand{\setS}{\mathcal{S}}

\newcommand\D{\mathbb{D}}

\begin{document}

\title{Heisenberg's uncertainty principle in the sense of Beurling}

\author{Haakan~Hedenmalm}
\address{Hedenmalm: Department of Mathematics\\
KTH Royal Institute of Technology\\
S--10044 Stockholm\\
Sweden}

\email{haakanh@math.kth.se}

%\author{Pekka~J.~Nieminen}

%\address{Nieminen: Department of Mathematics and Statistics\\
%University of Helsinki\\
%Box 68\\
%FI--00014 Helsinki\\
%Finland}

%\email{pjniemin@cc.helsinki.fi}

\date{20 March 2012}

\subjclass[2010]{}

\thanks{The author was supported by the G\"oran Gustafsson
Foundation (KVA) and by Vetenskapsr\r{a}det (VR).}

\begin{abstract} 
We shed new light on Heisenberg's uncertainty principle in the sense of 
Beurling, by offering an essentially different proof which allows us to 
weaken the assumptions substantially. The new formulation is essentially
optimal, as can be seen from examples. 
%We recall that Beurling's result asserts that under minimal assumptions on
%the function $f$,
%\[
%\int_\R|f(x)\hat f(y)|\,\e^{2\pi|xy|}\diff x
%\diff y<+\infty\quad\Longrightarrow\quad f=0.
%\]
%Here, 
%\[
%\hat f(y)=\int_\R \e^{-2\pi \imag yt}f(t)\,\diff t,
%\]
%denotes the usual Fourier transform. 
The proof involves Fourier and Mellin transforms. We also introduce a version
which applies to two given functions. We also show how our method applies in
the higher dimensional setting.
\end{abstract}

\maketitle

\section{Introduction} 
\label{intro}

We will write 
\[
\hat f(y):=\lim_{T\to+\infty}\int_{-T}^{T}\e^{-2\pi\imag yt}f(t)\,\diff t,
\qquad y\in\R,
\]
for the Fourier transform of the function $f$, wherever the limit converges.
For $f\in L^1(\R)$ the integral converges absolutely, and $\hat f$ is
continuous on $\R$ with limit $0$ at infinity (the Riemann-Lebesgue lemma); 
writing ${\mathrm C}_0(\R)$ for the Banach space of all such functions, 
we are merely saying that $\hat f\in {\mathrm C}_0(\R)$ whenever 
$f\in L^1(\R)$. By extending work of Hardy \cite{Hardy}, Beurling 
(see \cite{Beu}, p. 372, and H\"ormander \cite{H}) obtained
a version of Heisenberg's uniqueness principle which is attractive for its 
simplicity and beauty. The assertion is the following. If $f\in L^1(\R)$
and 
\begin{equation}
\int_\R\int_\R|f(x)\hat f(y)|\,\e^{2\pi|xy|}\diff x\diff y<+\infty,
\label{eq-beuass1}
\end{equation}
then $f=0$ a.e. on $\R$. Trivially, $1\le\e^{2\pi|xy|}$, so that if 
$f\in L^1(\R)$ with \eqref{eq-beuass1}, then we must also have that
\[
\|f\|_{L^1(\R)}\|\hat f\|_{L^1(\R)}=\int_\R\int_\R|f(x)\hat f(y)|\,\diff x
\diff y<+\infty.
\]
We see that the assumption \eqref{eq-beuass1} presupposes that  
$f$ and $\hat f$ are both in $L^1(\R)$.  As a result, $f$ is in the space
$L^1(\R)\cap \mathrm{C}_0(\R)$, which is contained in $L^p(\R)$ for all $p$
with $1\le p\le+\infty$. 

Our analysis of Beurling's theorem is based on the observation that under 
\eqref{eq-beuass1}, the function
\begin{equation}
F(\lambda):= 
\int_\R\int_\R \bar f(x)\hat f(y)\,\e^{2\pi\imag \lambda xy}\diff x\diff y
\label{eq-functF}
\end{equation}
defines a bounded holomorphic function in the strip
\[
\setS:=\{z\in\C:\,|\im\, z|<1\}
\]
which extends continuously to the closed strip $\bar{\setS}$. 
Indeed, the complex
exponentials $\e^{2\pi\imag\lambda xy}$ are holomorphic in $\lambda$, and
we have
\[
|F(\lambda)|\le 
\int_\R\int_\R |f(x)\hat f(y)|\,\e^{-2\pi xy\,\im\, \lambda}\diff x\diff y
\le \int_\R\int_\R|f(x)\hat f(y)|\,\e^{2\pi|xy|}\diff x\diff y,\qquad
\lambda\in\bar\setS,
\]
from which the claim is immediate, by, e.g., uniform convergence.
Next, in view of the Fourier inversion theorem, 
\[
\int_\R\hat f(y)\,\e^{2\pi\imag\lambda xy}\diff y=f(\lambda x),\qquad 
x,\lambda\in\R,
\]
so the function $F(\lambda)$ given by \eqref{eq-functF} may be expressed in
the form
\begin{equation}
F(\lambda)= \int_\R \bar f(x)f(\lambda x)\,\diff x,\qquad \lambda\in\R.
\label{eq-functF2}
\end{equation}
It is easy to see that $F(\lambda)$ is continuous on $\R^\times$ since
$f\in L^2(\R)$. Here, $\R^\times$ is shorthand for $\R\setminus\{0\}$.
Moreover, let $\D:=\{z\in\C:\,|z|<1\}$ denote the open unit disk in the 
complex plane $\C$, and let $\bar\D$ denote its closure (the closed unit disk).
We let $\diff A$ denote the area element in $\C$.
%$\T=\partial\D$ is the unit circle. 
%Also, let $\R^\times$ be shorthand for $\R\setminus\{0\}$, and let 
%$\diff A$ denote the area element in $\C$.

\begin{thm}
Suppose $f\in L^2(\R)$, and let $F(\lambda)$ be given by 
\eqref{eq-functF2} for $\lambda\in\R^\times$. 
%Let $\setE$ be a closed subset of logarithmic capacity $0$, which contains 
%the points $\pm\imag$. 
Suppose that $F(\lambda)$ has a holomorphic extension to a neighborhood of 
$\bar\D\setminus\{\pm\imag\}$, such that
\[
%\liminf_{\D\ni\lambda\to\pm\imag}|\lambda^2+1|^{1/2}|F(\lambda)|=0,
\int_\D|F(\lambda)|^2|\lambda^2+1|\,\diff A(\lambda)<+\infty.
\]
Then 

{\rm (a)} 
$F(\lambda)\equiv c_0(1+\lambda^2)^{-1/2}$ for some constant $c_0\ge0$, and

{\rm(b)} if, in addition, we have $c_0=0$, then $F(\lambda)\equiv0$, and 
consequently $f=0$ a.e. 
\label{thm-1}
\end{thm}

In comparison with Beurling's result, we assume analytic continuation of 
$F(\lambda)$ to a much smaller set, and the a priori assumption that 
$f\in L^2(\R)$ is weaker. Also, in Beurling's setting, the weighted square 
integrability of $F(\lambda)$ is trivially fulfilled because the function 
$F(\lambda)$ is then bounded on the strip $\setS$, which also shows that the 
infimum is $0$ under the heading (b) above. 
To see what case (a) of Theorem \ref{thm-1} means for the function $f$, 
we introduce the Mellin transform $\Mop_0$ as follows:
\begin{equation}
\Mop_0[f](\tau):=\int_{\R^\times}|x|^{-\frac12+\imag\tau}f(x)\,\diff x,\qquad
\tau\in\R.
\label{eq-Mellin0}
\end{equation}
%and 
%\[
%\Mop_1[f](\tau):=\int_{\R^\times}|x|^{-\frac12+\imag\tau}\mathrm{sgn}(x)\,
%f(x)\,\diff x,
%\]
%where $\mathrm{sgn}(x)=x/|x|$. 
%The $L^2$ theory for the Mellin transform is
%analogous to that of the Fourier transform (the Mellin transform is associated
%with the multiplicative structure, while the Fourier transform is related
%with the additive structure). We remark that the multiplicative group 
%$\R^\times$ is  isomorphic to the additive group $\R\times\mathbb{Z}_2$, 
%where $\mathbb{Z}_2=\mathbb{Z}/2\mathbb{Z}$.

\begin{thm}
Suppose $f\in L^2(\R)$, and let $F(\lambda)$ be given by 
\eqref{eq-functF2} for $\lambda\in\R^\times$. 
Then $F(\lambda)\equiv c_0(1+\lambda^2)^{-1/2}$ holds for some constant 
$c_0\ge0$ if and only if $f$ is even (i.e., $f(-x)=f(x)$ holds a.e.), and
\[
|\Mop_0[f](\tau)|=\frac{\sqrt{c_0}}{\pi^{1/4}}\,|\Gamma(\tfrac14+
\tfrac{\imag}{2}\tau)|,\qquad \tau\in\R.
\]
\label{thm-2}
\end{thm}

\begin{rem}
To better appreciate how much weaker the assumptions of Theorem \ref{thm-1}
are compared with those of Beurling's result, we may consider the assertion (a)
of Theorem \ref{thm-1} (a) combined with Theorem \ref{thm-2}.
The Mellin transform $\Mop_0[f]$ is then determined only to its modulus, which 
leads to the existence of a multitude of functions $f$ which solve 
\eqref{eq-functF2} for the given $F(\lambda)=c_0(1+\lambda^2)^{-1/2}$; e.g., 
the linear span of all such $f$ is infinite-dimensional. 
This contrasts with the analogues of Beurling's theorem where the constant 
(or polynomial) multiples of a Gaussian $\e^{-\alpha x^2}$ (with $\alpha>0$)
are the only solutions \cite{BDJ}.
\end{rem}

\begin{rem}
It is of interest to analyze the sharpness of the above results (Theorems 
\ref{thm-1} and \ref{thm-2}). Let us look at the example  
\[
f(x)=\e^{-\pi\beta x^2},
\]
where $\mathrm{Re}\,\beta>0$. Then $f\in L^2(\R)$, and the associated 
function $F(\lambda)$ is
\[
F(\lambda)=\int_\R \bar f(x)f(\lambda x)\diff x=\bar\beta^{-1/2}\bigg(
1+\frac{\beta}{\bar\beta}\lambda^2\bigg)^{-1/2}.
\]
This function $F(\lambda)$ is holomorphic in $\D$ but it possesses two square 
root branch points at the roots of $\lambda^2=-\bar\beta/\beta$.  These
roots lie on the unit circle $\mathbb{T}=\{z\in\C:\,|z|=1\}$. This means that
permitting just two such square root branch points along $\mathbb{T}$ in the 
formulation of Theorem \ref{thm-1} already falsifies the assertion of the 
theorem.   
In particular, if, in the formulation of Theorem \ref{thm-1}, the unit 
disk $\D$ is replaced by a proper open convex subset, then the 
conclusion of the theorem is no longer valid. 
\end{rem}

\section{A family of bilinear forms}

Let us consider the bilinear forms
\begin{equation}
\Bop[f,g](\lambda):=\int_\R f(t)\,g(\lambda t)\,\diff t,\qquad \lambda\in
\R^\times,
\label{eq-bilin0}
\end{equation}
for $f,g\in L^2(\R)$. The function $\Bop[f,g]$ is then continuous on 
$\R^\times$. It has the symmetry property
\begin{equation}
\Bop[f,g](\lambda)=\frac{1}{|\lambda|}\Bop[g,f]\bigg(\frac{1}{\lambda}\bigg),
\qquad \lambda\in\R^\times,
\label{bilin-symm}
\end{equation}
as we see by an elementary change of variables.
It also enjoys the complex conjugation symmetry
\begin{equation}
\overline{\Bop[f,g]}(\lambda)=\Bop[\bar f,\bar g](\lambda),
\qquad \lambda\in\R^\times.
\label{bilin-symm1.1}
\end{equation}
It is well-known that the multiplicative convolution
\[
f_1\circledast f_2(x):=\int_{\R^\times}f_1(t)\,f_2\bigg(\frac{x}{t}\bigg)\,
\frac{\diff t}{|t|},
\]
understood in the sense of Lebesgue, is commutative 
(i.e., $f_1\circledast f_2=f_2\circledast f_1$). The relationship with the
above bilinear forms $\Bop[f,g](\lambda)$ is 
\[
\Bop[f,g](\lambda)=g\circledast \tilde f(\lambda)=
\frac{1}{|\lambda|}\,f\circledast 
\tilde g\bigg(\frac{1}{\lambda}\bigg),\qquad \lambda\in\R^\times, 
\]
where 
\[
\tilde f(t):=\frac{1}{|t|}\,f\bigg(\frac{1}{t}\bigg),\quad
\tilde g(t):=\frac{1}{|t|}\,g\bigg(\frac{1}{t}\bigg).
\]

\section{The proofs of the first set of theorems}
\label{sec-proof}

\begin{proof}[Proof of Theorem \ref{thm-1}]
A comparison of \eqref{eq-functF2} and \eqref{eq-bilin0} reveals that 
$F(\lambda)=\Bop[\bar f,f](\lambda)$ for $\lambda\in\R^\times$.
In view of \eqref{bilin-symm} and \eqref{bilin-symm1.1}, $F(\lambda)$ has the 
symmetry property
\begin{equation}
F(\lambda)=\frac{1}{|\lambda|}\,\bar F\bigg(\frac{1}{\lambda}\bigg),\qquad
\lambda\in\R^\times.
\label{eq-symm1}
\end{equation}
Let $J(\lambda)$ be the function
\[
J(\lambda):=\sqrt{1+\lambda^2},
\]
which defines a single-valued holomorphic function in the slit complex plane
$\C\setminus\imag(\R\setminus]\!-\!1,1[)$ with value $1$ at $\lambda=0$. 
Next, we consider the function
$\Phi:=FJ$, which is a well-defined and continuous along $\R$, while it 
defines a holomorphic function in (a neighborhood of) 
$\bar\D\setminus\{\pm\imag\}$. Along the real line, we have, in view of 
\eqref{eq-symm1}, 
\begin{multline}
\Phi(\lambda)=F(\lambda)J(\lambda)=\frac{1}{|\lambda|}\,J(\lambda)\bar F\bigg(
\frac{1}{\lambda}\bigg)=\frac{\sqrt{1+\lambda^2}}{|\lambda|}\,\bar F\bigg(
\frac{1}{\lambda}\bigg)\\
=\sqrt{1+\frac{1}{\lambda^2}}\,\bar F\bigg(
\frac{1}{\lambda}\bigg)=\bar\Phi\bigg(\frac{1}{\lambda}\bigg)=
\bar\Phi\bigg(\frac{1}{\bar\lambda}\bigg),\qquad\lambda\in
\R^\times.
\label{eq-symm2}
\end{multline}
As a consequence, $\Phi$ is real-analytic on $\R$, and has two holomorphic
extensions, one to (a neighborhood of) $\bar\D\setminus\{\pm\imag\}$, and the 
other to (a neighborhood of) $\bar\D_e
\setminus\{\pm\imag\}$; here, $\bar\D_e:=\C\setminus\D$ is the closed exterior 
disk.
These two holomorphic continuations must then coincide. So, we see that
$\Phi$ extends to a holomorphic function in $\C\setminus\{\pm\imag\}$, which
is bounded in a neighborhood of infinity, by inspection of \eqref{eq-symm2}.
The integrability assumption of the theorem says that 
\[
\int_\D|\Phi(\lambda)|^2\diff A(\lambda)<+\infty,
\]
and the symmetry property \eqref{eq-symm2} gives the corresponding 
integrability in the exterior disk $\D_e=\C\setminus\bar\D$:
\[
\int_{\D_e}|\Phi(\lambda)|^2\frac{\diff A(\lambda)}{|\lambda|^4}<+\infty.
\] 
In particular, $\Phi$ is square area-integrable in a neighborhood of 
$\{\pm\imag\}$. But then $\Phi$ extends
holomorphically across $\pm\imag$ (one explanation among many:
a two-point set has logarithmic capacity $0$, see \cite{C}). 
Now $\Phi$ is entire and bounded, so Liouville's theorem tells us that 
$\Phi$ is constant: $\Phi(\lambda)\equiv c_0$. That $c_0\ge0$ follows from
\begin{equation*}
c_0=\Phi(1)=J(1)F(1)=\sqrt{2}\int_\R\bar f(x)\,f(x)\diff x=
\sqrt{2}\int_\R|f(x)|^2\diff x\ge0.
%\label{eq-c0}
\end{equation*}
This gives us the first assertion as well as the second. 
The proof is complete.
\end{proof}

\begin{proof}[Proof of Theorem \ref{thm-2}]
We need to show that if
\[
F(\lambda)=\int_\R \bar f(x)f(\lambda x)\diff x\equiv c_0(1+\lambda^2)^{-1/2},
\]
on $\R^\times$, then the Mellin transform $\Mop_0[f]$ has the indicated form.
We need the complementary Mellin transform
\[
\Mop_1[f](\tau):=\int_{\R^\times}|x|^{-\frac12+\imag\tau}\mathrm{sgn}(x)\,
f(x)\,\diff x,
\]
as well. Here, we write $\mathrm{sgn}(x)=x/|x|$. 
The $L^2$ theory for the Mellin transforms is
analogous to that of the Fourier transform (the Mellin transform is associated
with the multiplicative structure, while the Fourier transform is related
with the additive structure). We remark that the multiplicative group 
$\R^\times$ is  isomorphic to the additive group $\R\times\mathbb{Z}_2$, 
where $\mathbb{Z}_2=\mathbb{Z}/2\mathbb{Z}$. By symmetry, we see that
$\Mop_1[F](\tau)\equiv0$, while a computation reveals that
\[
\Mop_0[F](\tau)=c_0\int_{\R^\times}|\lambda|^{-\frac12+\imag\tau}
(1+\lambda^2)^{-1/2}\diff\lambda=\frac{c_0}{\sqrt{\pi}}\,|\Gamma(\tfrac14+
\tfrac{\imag}{2}\tau)|^2,
\]
If we apply the Mellin transforms $\Mop_0,\Mop_1$ to \eqref{eq-functF2},
we find that $\Mop_1[f]\equiv0$ and that
\[
\Mop_0[F](\tau)=|\Mop_0[f](\tau)|^2=\Mop_0[F](\tau).
\]
Here, the natural way to verify the right-hand side equality is to apply
the inverse Mellin transform to the two sides. The assertion that 
$\Mop_1[f]\equiv0$ holds if and only if $f$ is an even function.
The proof is complete.
\end{proof}

\section{An extension involving two functions}

We consider two functions $f,g\in L^2\R)$, and introduce the functions
\begin{equation}
F_1(\lambda):=\int_\R \bar f(x)g(\lambda x)\diff x,\qquad
F_2(\lambda):=\int_\R \bar g(x)f(\lambda x)\diff x.
\label{eq-functF12}
\end{equation}
We quickly observe that if $f$ is even and $g$ is odd, then $F_1(\lambda)=
F_2(\lambda)=0$ on $\R^\times$. The same conclusion holds if $f$ is odd and
$g$ is even. This means that we cannot hope to claim that one of the functions
$f,g$ must vanish. But sometimes this combination of even and odd is the only
obstruction, as we shall see.

\begin{thm}
Suppose $f,g\in L^2(\R)$, and let $F_j(\lambda)$ be given by 
\eqref{eq-functF2} for $\lambda\in\R^\times$ and $j=1,2$. 
%Let $\setE$ be a closed subset of logarithmic capacity $0$, which contains 
%the points $\pm\imag$. 
Suppose that both $F_j(\lambda)$ have a holomorphic extensions to $\D$
such that
\[
\int_\D|F_j(\lambda)|^2|\lambda^2+1|\,\diff A(\lambda)<+\infty,\qquad j=1,2.
\]
Suppose, moreover, that one of the functions, say $F_1$, has a holomorphic 
extension to a neighborhood of $\bar\D\setminus\{\pm\imag\}$.
Then 

{\rm (a)} 
$F_j(\lambda)\equiv c_j(1+\lambda^2)^{-1/2}$ for $j=1,2$, for some constants 
$c_1,c_2\in\C$ with $c_2=\bar c_1$, and

{\rm(b)} if, in addition, we have $c_1=0$, then $F_1(\lambda)\equiv 
F_2(\lambda)\equiv0$. 
\label{thm-3}
\end{thm}

The application of the Mellin transforms leads to the following result.

\begin{thm}
Suppose $f,g\in L^2(\R)$, and let $F_1(\lambda)$ be given by 
\eqref{eq-functF12} for $\lambda\in\R^\times$. 
Then $F_1(\lambda)\equiv c_1(1+\lambda^2)^{-1/2}$ holds for some constant 
$c_1\in\C$ if and only if 
\[
\overline{\Mop_1[f](\tau)}
\Mop_1[g](\tau)=0,\qquad \mathrm{a.e.}\,\, \tau\in\R,
\]
and
\[
\overline{\Mop_0[f](\tau)}\,\Mop_0[g](\tau)=\frac{c_1}{\sqrt{\pi}}
\,|\Gamma(\tfrac14+\tfrac{\imag}{2}\tau)|^2,\qquad \mathrm{a.e.}\,\,\tau\in\R.
\]
\label{thm-4}
\end{thm}

The assertion of Theorem \ref{thm-4} gives a very precise answer as to what
$f,g$ can be in the setting of Theorem \ref{thm-3}. It may however at times
be difficult to see what the conditions actually say when $f,g$ are explicitly
given. So we will explain a couple of cases when we can be more precise. 
The {\em support} of a function $f\in L^2(\R)$ -- written $\mathrm{supp}\,f$ 
-- is the intersection of all closed sets $E\subset\R$ such that $f=0$ a.e. 
on $\R\setminus E$. 
Let us agree to say that a function $f\in L^2(\R)$ has {\em dilationally 
one-sided support} if (i) $\mathrm{supp}\,f$ is bounded in $\R$, or if (ii) 
$\mathrm{supp}\,f\subset\R^\times=\R\setminus\{0\}$.  

\begin{thm}
Suppose $f,g\in L^2(\R)$, and let $F_1(\lambda)$ be given by 
\eqref{eq-functF12} for $\lambda\in\R^\times$, and suppose that 
$F_1(\lambda)\equiv0$. If $f$ has dilationally one-sided support, then either:
{\rm(a)} $f$ is even and $g$ is odd, {\rm(b)} $f$ is odd and $g$ is 
even, or {\rm(c)} $f=0$ a.e. or $g=0$ a.e. 
\label{thm-4'}
\end{thm}

\section{Proofs of the theorems involving two functions}

\begin{proof}[Proof of Theorem \ref{thm-3}]
A comparison of \eqref{eq-functF2} and \eqref{eq-bilin0} reveals that 
$F_1(\lambda)=\Bop[\bar f,g](\lambda)$ and $F_2(\lambda)=
\Bop[\bar g,f](\lambda)$ for $\lambda\in\R^\times$.
In view of \eqref{bilin-symm} and \eqref{bilin-symm1.1}, $F_j(\lambda)$, 
for $j=1,2$, have the symmetry property
\begin{equation}
F_1(\lambda)=\frac{1}{|\lambda|}\,\bar F_2\bigg(\frac{1}{\lambda}\bigg),\qquad
\lambda\in\R^\times.
\label{eq-symm1'}
\end{equation}
If we put $\Phi_j:=F_jJ$, where $J(\lambda)=(1+\lambda^2)^{1/2}$ as before,
then \eqref{eq-symm1'} says that
\[
\Phi_1(\lambda)=\bar\Phi_2\bigg(\frac{1}{\bar\lambda}\bigg),\qquad\lambda\in
\R^\times.
\]
The given assumptions on $F_1,F_2$ show that $\Phi_1$ has a holomorphic 
extension to $\C\setminus\{\pm\imag\}$, which is area-$L^2$ integrable locally 
around $\{\pm\imag\}$. As a consequence, the singularities at $\{\pm\imag\}$
are removable (see, e.g., \cite{C}), and Liouville's theorem tells us that
$\Phi_1$ is constant. The remaining assertions are easy consequences of this.
\end{proof}

\begin{proof}[Proof of Theorem \ref{thm-4}]
The proof is immediate by taking the Mellin transforms, as in the proof
of Theorem \ref{thm-2}. We omit the details.
\end{proof}

\begin{proof}[Proof of Theorem \ref{thm-4'}]
In view of Theorem \ref{thm-4}, we have that
\begin{equation}
\overline{\Mop_0[f](\tau)}\,\Mop_0[g](\tau)=
\overline{\Mop_1[f](\tau)}\Mop_1[g](\tau)=0,\qquad \mathrm{a.e.}\,\, \tau\in\R.
\label{eq-prodrel}
\end{equation}
The assumption that $f$ has dilationally one-sided support means in terms of
Mellin transforms that the functions $\Mop_j[f]$, $j=0,1$, both extend to a 
function in $H^2$ of either the upper or the lower half-plane. In any case,
Privalov's theorem guarantees that for a given $j\in\{0,1\}$, either 
$\Mop_j[f]=0$ a.e. on $\R$, or $\Mop_j[f]\ne0$ a.e. on $\R$. This leaves us
with four different possibilities. 

CASE 1.  $\Mop_0[f]=0$ a.e.and $\Mop_1[f]=0$ a.e.. Then $f=0$ a.e.is immediate,
so we find purselves in the setting of (c).

CASE 2. $\Mop_0[f]=0$ a.e. and $\Mop_1[f]\ne0$ a.e. Then \eqref{eq-prodrel}
gives that $\Mop_1[g]=0$ a.e., so that $f$ is odd and
$g$ is even, and we are in the setting of (b). 

CASE 3. $\Mop_0[f]\neq0$ a.e. and $\Mop_1[f]=0$ a.e. Then \eqref{eq-prodrel}
gives that $\Mop_0[g]=0$ a.e., and we conclude that $f$ is even and $g$ is odd,
and we are in the setting of (a). 

CASE 4.$\Mop_0[f]\neq0$ a.e. and $\Mop_1[f]\ne0$ a.e.  Then 
\eqref{eq-prodrel} shows that 
$\Mop_0[g]=\Mop_1[g]=0$ a.e., so that $g=0$ a.e., and we are in the setting 
of (c). 
\end{proof}

\section{A higher dimensional analogue}

We present an analogue of Theorem \ref{thm-1} for $\R^n$, $n=1,2,3,\ldots$;
$\diff\vol_n$ is volume measure in $\R^n$. 
For $f\in L^2(\R^n)$, let $F(\lambda)$ be the function
\begin{equation}
F(\lambda)= \int_{\R^n} \bar f(x)f(\lambda x)\,\diff\vol_n(x),
\qquad \lambda\in\R^\times.
\label{eq-functF22}
\end{equation}
This function arises if we write 
\[
F(\lambda)=\int_{\R^n}\int_{\R^n} \bar f(x)\hat f(y)
\e^{2\pi\lambda \imag \langle x,y\rangle}\diff\vol_n(x)\diff\vol_n(y),
\]
where 
\[
\hat f(y)=\int_{\R^n}\e^{-2\pi\imag \langle x,y\rangle}
\diff\vol_n(x)
\]
is the usual Fourier transform. Here, 
\[
\langle x,y\rangle=x_1y_1+\cdots+x_ny_n,\qquad x=(x_1,\ldots, x_n),\,\,\,
y=(y_1,\ldots, y_n),
\]
is the usual inner product in $\R^n$.
 
\begin{thm}
Suppose $f\in L^2(\R^n)$, and let $F(\lambda)$ be given by 
\eqref{eq-functF22} for $\lambda\in\R^\times$. 
Suppose that $F(\lambda)$ has a holomorphic extension to a neighborhood of 
$\bar\D\setminus\{\pm\imag\}$, such that
\[
\int_\D|F(\lambda)|^2|\lambda^2+1|^n\,\diff A(\lambda)<+\infty.
\]
Then 

{\rm (a)} 
$F(\lambda)\equiv c_0(1+\lambda^2)^{-n/2}$ for some constant $c_0\ge0$, and

{\rm(b)} if, in addition, we have $c_0=0$, then $F(\lambda)\equiv0$, and 
consequently $f=0$ a.e.
\label{thm-5}
\end{thm}

\section{Proof of the higher dimensional analogue}

\begin{proof}[Proof of Theorem \ref{thm-5}]
We indicate what differs from the case $n=1$, which is covered by the proof
of Theorem \ref{thm-1}. 
An exercise involving a change of variables shows that $F(\lambda)$ has the 
symmetry 
property
\begin{equation}
F(\lambda)=\frac{1}{|\lambda|^n}\,\bar F\bigg(\frac{1}{\lambda}\bigg),\qquad
\lambda\in\R^\times.
\label{eq-symm1''}
\end{equation}
Let $J_n(\lambda)$ be the function
\[
J_n(\lambda):=(1+\lambda^2)^{n/2}.
\]
Next, we consider the function $\Phi:=FJ_n$, which is a well-defined and 
continuous along $\R$, while it 
defines a holomorphic function in (a neighborhood of) 
$\bar\D\setminus\{\pm\imag\}$. Along the real line, we have, in view of 
\eqref{eq-symm1''}, 
\begin{multline}
\Phi(\lambda)=F(\lambda)J(\lambda)=\frac{1}{|\lambda|^n}\,
J(\lambda)\bar F\bigg(\frac{1}{\lambda}\bigg)=
\frac{(1+\lambda^2)^{n/2}}{|\lambda|^n}\,\bar F\bigg(
\frac{1}{\lambda}\bigg)\\
=\bigg(1+\frac{1}{\lambda^2}\bigg)^{n/2}\,\bar F\bigg(
\frac{1}{\lambda}\bigg)=\bar\Phi\bigg(\frac{1}{\lambda}\bigg)=
\bar\Phi\bigg(\frac{1}{\bar\lambda}\bigg),\qquad\lambda\in
\R^\times.
\label{eq-symm2''}
\end{multline}
As a consequence of the assumptions, $\Phi$ extends to a holomorphic function 
in $\C\setminus\{\pm\imag\}$, which is bounded in a neighborhood of infinity, 
by inspection of \eqref{eq-symm2''}.
The integrability assumption of the theorem says that $\Phi$ is area-$L^2$
integrable near $\{\pm\imag\}$, so that the singularities at $\pm\imag$ are
removable. Liouville's theorem tells us that $\Phi$ is constant: 
$\Phi(\lambda)\equiv c_0$. That $c_0\ge0$ follows from
\begin{equation*}
c_0=\Phi(1)=J_n(1)F(1)=2^{n/2}\int_\R\bar f(x)f(x)\diff \vol_n(x)=
2^{n/2}\int_\R|f(x)|^2\diff\vol_n(x)\ge0.
%\label{eq-c0''}
\end{equation*}
This gives us the first assertion as well as the second. 
The proof is complete.
\end{proof}

%\begin{bibsection} 
%\begin{biblist} 

% \end{biblist} 
% \end{bibsection} 

\end{document}